\documentclass[12pt]{amsart}
\usepackage{graphicx}
\usepackage[headings]{fullpage}
\usepackage{amssymb,epic,eepic,epsfig,amsbsy,amsmath,amscd}
\numberwithin{equation}{section}
                        \theoremstyle{plain}
\usepackage{mathrsfs}

\newcommand{\psdraw}[2]
         {\begin{array}{c} \hspace{-1.3mm}
         \raisebox{-4pt}{\psfig{figure=#1.eps,width=#2}}
         \hspace{-1.9mm}\end{array}}
         
\newcommand\no[1]{}

\newtheorem{theorem}{Theorem}[section]

\newtheorem{lemma}[theorem]{Lemma}
\newtheorem{corollary}[theorem]{Corollary}
\newtheorem{proposition}[theorem]{Proposition}

\theoremstyle{definition}
\newtheorem{remark}[theorem]{Remark}

\newcommand{\C}{\Bbb C}

\newcommand{\tr}{{\mathrm{tr}\,}}
\newcommand{\la}{\langle}
\newcommand{\ra}{\rangle}

\def\BC{\mathbb C}
\def\C{\mathbb C}

\def\BZ{\mathbb Z}
\def\BR{\mathbb R}

\def\CL{\mathcal L}

\def\fb{\mathfrak b}

\def\la{\langle}
\def\ra{\rangle}

\def\ve{\varepsilon}
\def\be { \begin{equation} }
\def\ee { \end{equation} }

\begin{document}

\title[The A-polynomial of twisted Whitehead links]
{The A-polynomial $2$-tuple of twisted Whitehead links}

\author{Anh T. Tran}

\begin{abstract} 
We compute the A-polynomial 2-tuple of twisted Whitehead links.
As applications, we determine canonical components of twisted Whitehead links and 
give a formula for the volume of twisted Whitehead link cone-manifolds.
\end{abstract}



\address{Department of Mathematical Sciences, The University of Texas at Dallas, 
Richardson, TX 75080, USA}
\email{att140830@utdallas.edu}

\maketitle

\section{Introduction}\label{sec:intro}

The A-polynomial of a knot in the 3-sphere $S^3$ was introduced by Cooper, Culler, Gillet, Long and Shalen \cite{CCGLS} in the 90's. 
It describes the $SL_2(\BC)$-character variety of the knot complement as viewed from the boundary torus. 
The A-polynomial carries a lot of information about the topology of the knot.
 For example, it distinguishes the unknot from other knots \cite{BoZ, DG} 
 and the sides of the Newton polygon of the A-polynomial give rise to incompressible surfaces in the knot complement \cite{CCGLS}. 
 
 The A-polynomial was generalized to links by Zhang \cite{Zh} about ten years later. For an $m$-component link in $S^3$, 
 Zhang defined a polynomial $m$-tuple link invariant called the A-polynomial $m$-tuple. 
 The A-polynomial $1$-tuple of a knot is nothing but the A-polynomial defined in \cite{CCGLS}.
 The A-polynomial $m$-tuple also caries important information about the topology of the link. 
 For example, it can be used to construct concrete examples of hyperbolic link manifolds with non-integral traces \cite{Zh}.
 
 Finding an explicit formula for the A-polynomial is a challenging problem. 
 So far, the A-polynomial has been computed for a few classes of knots including 
 two-bridge knots $C(2n, p)$ (with $1 \le p \le 5$) in Conway's notation \cite{HS, Pe, Ma, HL},  
 $(-2,3,2n+1)$-pretzel knots \cite{TY, GM}. 
 It should be noted that $C(2n,p)$ is the double twist knot $J(2n,-p)$ in the notation of \cite{HS}. 
 Moreover, $C(2n,1)$ is the torus knot $T(2,2n+1)$ and $C(2n,2)$ is known as a twist knot. 
 A cabling formula for the A-polynomial of a cable knot in $S^3$ has recently been given in \cite{NZ}. 
 Using this formula, Ni and Zhang \cite{NZ} has computed the A-polynomial of an iterated torus knot explicitly.
 
 In this paper we will compute the A-polynomial $2$-tuple for 
 a family of 2-component links called twisted Whitehead links. 
 As applications, we will determine canonical components of twisted Whitehead links and 
give a formula for the volume of twisted Whitehead link cone-manifolds. 
 For $k \ge 0$, the $k$-twisted Whitehead link $W_k$ is the 2-component link depicted in Figure \ref{Wh}. 
Note that $W_0$ is the torus link $T(2,4)$ and $W_1$ is the Whitehead link. 
Moreover, $W_k$ is the two-bridge link $C(2,k,2)$ in Conway's notation 
and is $\fb(4k+4,2k+1)$ in Schubert's notation.  
These links are all hyperbolic except for $W_0$. 
 
\begin{figure}[htpb] \label{Wh}
$$\psdraw{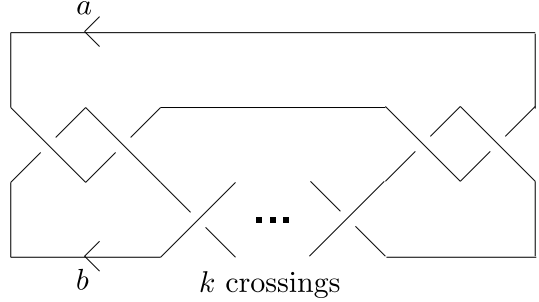}{2.75in}$$
\caption{The $k$-twisted Whitehead link $W_k$.}
\end{figure}

The A-polynomial $2$-tuple of the twisted Whitehead link $W_k$ is a polynomial $2$-tuple 
$[A_1(M,L), A_2(M,L)]$ given as follows. 

\begin{theorem} \label{thm:A-polynomial}
If $k=2n-1$ then $A_1(M,L) = (L-1)F(M,L)$ where
\begin{eqnarray*}
F(M,L) &=& \sum_{i=0}^n \left\{ \binom{n+1+i}{2i+1} - \binom{n-1+i}{2i+1} \right\} 
(M - M^{-1})^{2i} \left( \frac{L -1}{L+1} \right)^{2i} \\
&& + \, \sum_{i=0}^{n-1} \binom{n+i}{2i+1} (M + M^{-1}) (M - M^{-1})^{2i+1} \left( \frac{L-1}{L+1} \right)^{2i+1}.
\end{eqnarray*}

If $k=2n$ then $A_1(M,L) = (LM^2-1)G(M,L)$ where
\begin{eqnarray*}
G(M,L) &=& \sum_{i=0}^n \left\{ \binom{n+1+i}{2i+1} + \binom{n+i}{2i+1} \right\} 
(M - M^{-1})^{2i+1} \left( \frac{LM^2 -1}{LM^2+1} \right)^{2i+1} \\
&& + \, \sum_{i=0}^{n} \binom{n+i}{2i} (M + M^{-1}) (M - M^{-1})^{2i} \left( \frac{LM^2-1}{LM^2+1} \right)^{2i}.
\end{eqnarray*}

In both cases we have $A_1(M,L) = A_2(M,L)$.
\end{theorem}

It should be noted that in the definition of the A-polynomial $m$-tuple of an $m$-component link, 
we discard the irreducible component of the $SL_2(\BC)$-character variety 
containing only characters of reducible representations and 
therefore consider only the nonabelian $SL_2(\BC)$-character variety of the link group. 

For a link $\CL$ in $S^3$, let $E_{\CL} = S^3 \setminus \CL$ denote the link exterior. 
The link group of $\CL$ is defined to be the fundamental group $\pi_1(E_{\CL})$ of $E_{\CL}$. 
It is known that the link group of the twisted Whitehead link $W_k$ has a standard two-generator presentation 
of a two-bridge link group $\pi_1(E_{W_k}) = \la a,b \mid aw =wa \ra$, 
where $a,b$ are meridians depicted in Figure \ref{Wh} and 
$w$ is a word in the letters $a,b$. More precisely, we have
$w=(bab^{-1}a^{-1})^na(a^{-1}b^{-1}ab)^n$ if $k=2n-1$ and 
$w=(bab^{-1}a^{-1})^n bab (a^{-1}b^{-1}ab)^n$ if $k=2n$. 

For a presentation $\rho: \pi_1(E_{W_k}) \to SL_2(\BC)$ we let $x,y,z$ denote the traces 
of the images of $a,b,ab$ respectively. We also let $v$ denote the trace of the image of $bab^{-1}a^{-1}$. 
Explicitly, we have 
$v = x^2 + y^2 + z^2 - xyz -2.$ It was shown in \cite{Tran-character} that 
the nonabelian $SL_2(\BC)$-character variety of $W_k$ has exactly $\lfloor \frac{k}{2} \rfloor $ irreducible components. 
For a hyperbolic link, there are distinguished components of the $SL_2(\BC)$-character variety called the canonical components. 
They contain important information about the hyperbolic structure of the link. 
By using the formula of the A-polynomial in Theorem \ref{thm:A-polynomial}, we can determine 
the canonical component of the hyperbolic twisted Whitehead link $W_k$ (with $k \ge 1$) as follows. 

\begin{theorem} \label{thm:canonical}
If $k=2n-1$ then the canonical component of $W_k$ is the zero set of the polynomial 
$(xy - v z) S_{n-1}(v) - (xy - 2z)S_{n-2}(v)$.

If $k=2n$ then the canonical component of $W_k$ is the zero set of the polynomial 
$z S_{n}(v) - (xy -  z) S_{n-1}(v)$.
\end{theorem}

Here the $S_k(v)$ are the Chebychev polynomials of the second kind defined by 
$S_0(v)=1$, $S_1(v)=v$ and $S_{k}(v) = v S_{k-1}(v) - S_{k-2}(v)$ 
for all integers $k$. Explicitly, we have $$S_k(v) = \sum_{0 \le i \le k/2} (-1)^i \binom{k-i}{k-2i} v^{k-2i}$$ for $k \ge 0$, 
$S_k(v) = - S_{-k-2}(v)$ for $k \le -2$, and $S_{-1}(v) =0$. 

For a hyperbolic link $\CL \subset S^3$, let $\rho_{\text{hol}}$ be a holonomy representation of $\pi_1(E_\CL)$ into $PSL_2(\BC)$. Thurston \cite{Th} showed that $\rho_{\text{hol}}$ can be deformed into an $m$-parameter family $\{\rho_{\alpha_1, \cdots, \alpha_m}\}$ of representations to give a corresponding family $\{E_\CL(\alpha_1, \cdots, \alpha_m)\}$ of singular complete hyperbolic manifolds, where $m$ is the number of components of $\CL$. In this paper we will consider only the case where all of $\alpha_i$'s are equal to a single parameter $\alpha$. In which case we also denote $E_\CL(\alpha_1, \cdots, \alpha_m)$ by $E_\CL(\alpha)$. These $\alpha$'s and $E_\CL(\alpha)$'s are called the cone-angles and hyperbolic cone-manifolds of $\CL$, respectively. 

We consider the complete hyperbolic structure on a link complement as the cone-manifold structure of cone-angle zero. It is known that for a two-bridge knot or link $\CL$ there exists an angle $\alpha_\CL \in [\frac{2\pi}{3},\pi)$ such that $E_\CL(\alpha)$ is hyperbolic for $\alpha \in (0, \alpha_\CL)$, Euclidean for $\alpha =\alpha_\CL$, and spherical for $\alpha \in (\alpha_\CL, \pi)$  \cite{HLM, Ko, Po, PW}. Explicit volume formulas for hyperbolic cone-manifolds of knots were known for $4_1$ \cite{HLM, Ko, Ko2, MR}, $5_2$ \cite{Me},  twist knots \cite{HMP} and two-bridge knots $C(2n,3)$ \cite{HL}. Recently, the volume of double twist knot cone-manifolds has been computed in \cite{Tr-volume}. 
We should remark that a formula for the volume of the cone-manifold of $C(2n,4)$ has just been given in \cite{HLMR-knot}. However, with an appropriate change of variables, this formula has already obtained in \cite{Tr-volume}, since $C(2n,4)$ is the double twist knot $J(2n,-4)$. 

In this paper we are interested in hyperbolic cone-manifolds of links. 
Explicit volume formulas for hyperbolic cone-manifolds of links have been only known for $5_1^2$ \cite{MV}, $6_2^2$ \cite{Me-link},  $6_3^2$ \cite{DMM} and $7_3^2$ \cite{HLMR-link}. From the proof of Theorem \ref{thm:A-polynomial}, we can compute the volume of the hyperbolic cone-manifold of the twisted Whitehead link $W_k$ as follows. 

Let
$$
R_{W_k}(s,z) = \begin{cases} x^2S_{n-1}(v) - z S_n(v)- (x^2-z)S_{n-2}(v) &\mbox{if } k=2n-1 \\ 
z S_n(v) - (x^2-z)S_{n-1}(v) & \mbox{if } k=2n, \end{cases}
$$
where $x = s+s^{-1}$ and $v = 2x^2+z^2-x^2z-2$.

\begin{theorem} \label{thm:volume}
For $\alpha \in (0, \alpha_{W_k})$ we have
$$\emph{Vol} \, E_{W_k}(\alpha) = \int_{\alpha}^{\pi} 2\log \left| \frac{z - (s^{-2}+1)}{z- (s^2+1) } \right| d\omega$$ 
where $s=e^{i\omega/2}$ and $z$ (with $\text{Im} \, z \ge 0$) satisfy $R_{W_k}(s,z) =0$.
\end{theorem}

Note that the above volume formula for the hyperbolic cone-manifold $E_{W_k}(\alpha)$ depends on the choice of a root $z$ of  $R_{W_k}(e^{i\omega/2},z) =0$. In practice, we choose the root $z$ which gives the maximal volume.

As a direct consequence of Theorem \ref{thm:volume}, we obtain the following. 

\begin{corollary}
The hyperbolic volume of the $r$-fold cyclic covering over the twisted Whitehead link $W_k$, with $r \ge 3$, 
is given by the following formula
$$r \emph{Vol} \, E_{W_k}(\frac{2\pi}{r}) = r \int_{\frac{2\pi}{r}}^{\pi} 2\log \left| \frac{z - (s^{-2}+1)}{z- (s^2+1) } \right| d\omega$$ 
where $s=e^{i\omega/2}$ and $z$ (with $\text{Im} \, z \ge 0$) satisfy $R_{W_k}(s,z) =0$.
\end{corollary} 

\begin{remark}
The link $7_3^2$ is the twisted Whitehead link $W_3$. 
With an appropriate change of variables, we obtain the volume formula for the hyperbolic cone-manifold of $7_3^2$ in \cite[Theorem 1.1]{HLMR-link}  by taking $k=3$ in Theorem \ref{thm:volume}.
\end{remark}

The paper is organized as follows. In Section \ref{sec:prelim} we review the definition of the A-polynomial $m$-tuple of an $m$-component link in $S^3$. In Section \ref{sec:nonab-repn} we review the nonabelian $SL_2(\BC)$-representations of a two-bridge link and compute them explicitly for twisted Whitehead links. We also give proofs of Theorems \ref{thm:A-polynomial}--\ref{thm:volume} in Section \ref{sec:nonab-repn}.

\section{The A-polynomial} \label{sec:prelim}

\subsubsection{Character varieties}

The set of characters of representations of a finitely generated group $G$ into $SL_2(\BC)$ is known to be a complex algebraic set, called the character variety of $G$ and denoted by $\chi(G)$ (see \cite{CS, LM}). For a manifold $Y$ we also use $\chi(Y)$ to denote $\chi(\pi_1(Y))$. Suppose  $G=\BZ^2$, the free abelian group with 2 generators.
Every pair  of generators $\mu,\lambda$ will define an isomorphism
between $\chi(G)$ and $(\BC^*)^2/\tau$, where $(\BC^*)^2$ is the
set of non-zero complex pairs $(M,L)$ and $\tau: (\BC^*)^2 \to (\BC^*)^2$ is the involution
defined by $\tau(M,L):=(M^{-1},L^{-1})$, as follows. Every representation is
conjugate to an upper diagonal one, with $M$ and $L$ being the
upper left entry of $\mu$ and $\lambda$ respectively. The
isomorphism does not change if we replace $(\mu,\lambda)$ with
$(\mu^{-1},\lambda^{-1})$.

\subsubsection{The A-polynomial}

Suppose $\CL = K_1 \sqcup \dots \sqcup K_m$ be an $m$-component link in $S^3$. 
Let $E_\CL = S^3 \setminus \CL$ be the link exterior and $T_1, \dots, T_m$ 
the boundary tori of $E_\CL$ corresponding to $K_1, \dots, K_m$ respectively. 
Each $T_i$ is a torus whose fundamental group  is free abelian of rank
two. An orientation of $K_i$ will define a unique pair of an
oriented meridian $\mu_i$ and an oriented longitude $\lambda_i$ such that the linking
number between the longitude $\lambda_i$ and the knot $K_i$ is 0. The pair provides
an identification of $\chi(\pi_1(T_i))$ and $(\BC^*)_i^2/\tau_i$, where $(\BC^*)_i^2$ is the
set of non-zero complex pairs $(M_i,L_i)$ and $\tau_i$ is the involution
$\tau(M_i,L_i)=(M_i^{-1},L_i^{-1})$, which actually does not depend on the orientation of $K_i$.

The inclusion $T_i \hookrightarrow E_{\CL}$ induces the restriction
map
$$\rho_i : \chi(E_{\CL}) \longrightarrow \chi(T_i)\equiv (\BC^*)_i^2/\tau_i.$$
For each $\gamma \in \pi_1(E_{\CL})$ let $f_{\gamma}$ be the regular function on the character variety $\chi(E_{\CL})$ defined by $f_{\gamma}(\chi_{\rho})=(\chi_{\rho}(\gamma))^2-4$, where $\chi_\rho$ denotes the character of a representation 
$\rho: \pi_1(E_{\CL}) \to SL_2(\BC)$. Let $\chi_i(X)$ be the subvariety of $\chi(X)$ defined by $f_{\mu_j}=0,~f_{\lambda_j}=0$ for all $j \not= i.$
Let $Z_i$ be the image of $\chi_i(X)$ under 
$\rho_i$ and  $\hat Z_i \subset (\BC^*)_i^2$ the lift of $Z_i$ under the
projection $(\BC^*)_i^2 \to (\BC^*)_i^2/\tau_i$. It is known that the Zariski closure of
$\hat Z_i\subset (\BC^*)_i^2 \subset \BC_i^2$ in $\BC_i^2$ is an algebraic set
consisting of components of dimension 0 or 1 \cite{Zh}. The union of all the
1-dimension components is defined by a single polynomial $A'_i \in
\BZ[M_i,L_i]$ whose coefficients are co-prime. Note that $A'_i$ is defined up to $\pm 1$. 
It is also known that $A'_i$ always contains the factor $L_i-1$ coming from the characters of reducible representations, 
hence $A'_i=(L_i-1)A_i$ for some $A_i \in \BZ[M_i,L_i]$. 
As in \cite{Zh}, we will call $[A_1(M_1, L_1), \cdots, A_m(M_m, L_m)]$ the A-polynomial $m$-tuple of $\CL$. 
For brevity, we also write $A_i(M,L)$ for $A_i(M_i, L_i)$. We refer the reader to \cite{Zh} for properties of the A-polynomial $m$-tuple.

Recall that the Newton polygon of a two-variable polynomial $\sum {a_{ij}}M^i L^j$ is the convex hull in $\BR^2$ of the set
$\{(i,j) : a_{ij} \not= 0\}$. The slope of a side of the Newton polygon is called a boundary
slope of the polygon. The following proposition is useful for determining canonical components of hyperbolic links.

\begin{proposition} \label{prop:hyp}
\cite[Theorem 3(2)]{Zh}
Suppose $\CL \subset S^3$ is a hyperbolic $m$-component link. 
For each $j = 1, \cdots, m$, the factor $A_j^{\emph{can}}(M,L)$ of the A-polynomial $A_j(M,L)$ corresponding to a canonical component 
has an irreducible factor whose Newton polygon has at least two distinct boundary slopes. In particular, this irreducible factor contains at least 3 monomials in $M,L$.
\end{proposition}

\section{Proofs of Theorems \ref{thm:A-polynomial}--\ref{thm:volume}} \label{sec:nonab-repn}

In this section we review the nonabelian $SL_2(\BC)$-representations of a two-bridge link from \cite{Ri} and compute them explicitly for twisted Whitehead links. 
Finally, we give proofs of Theorems \ref{thm:A-polynomial}--\ref{thm:volume}.

\subsection{Two-bridge links} 

Two-bridge links are those links admitting a projection with only two maxima and two minima. 
The double branched cover of $S^3$ along a two-bridge link is a lens space $L(2p,q)$, 
which is obtained by doing a $2p/q$ surgery on the unknot. 
Such a two-bridge link is denoted by $\fb(2p,q)$. 
Here $q$ is an odd integer co-prime with $2p$ and $2p > |q| \ge 1$. 
It is known that $\fb(2p',q')$ is ambient isotopic to $\fb(2p,q)$ 
if and only if $p'=p$ and $q' \equiv q^{\pm 1} \pmod{2p}$, see e.g. \cite{BuZ}. 
The link group of the two-bridge link $\fb(2p,q)$ has a standard 
two-generator presentation $\pi_1(E_{\fb(2p,q)}) = \la a, b \mid wa = a w \ra$ 
where $a,b$ are meridians, $w=b^{\ve_1} a^{\ve_2} \cdots a^{\ve_{p-2}}b^{\ve_{2p-1}}$ and 
$\ve_i=(-1)^{\lfloor iq/(2p) \rfloor}$ for $1 \le i \le 2p-1$. Note that $\ve_i = \ve_{2p-i}$.

We now study representations of link groups into $SL_2(\BC)$. 
A representation is called nonabelian if its image is a nonabelian subgroup of $SL_2(\BC)$. 
Let $\CL = \fb(2p,q)$. 
Suppose $\rho: E_\CL \to SL_2(\C)$ is a nonabelian representation. 
Up to conjugation, we may assume that $$\rho(a) = \left[ \begin{array}{cc}
s_1 & 1 \\
0 & s_1^{-1} \end{array} \right] \quad \text{and} \quad 
\rho(b) = \left[ \begin{array}{cc}
s_2 & 0 \\
u & s_2^{-1} \end{array} \right]$$
where $(u, s_1, s_2) \in \C^3$ satisfies the matrix equation $\rho(aw) = \rho(wa)$. 

For any word $r$, we write $\rho(r) = \left[ \begin{array}{cc}
r_{11} & r_{12} \\
r_{21} & r_{22} \end{array} \right]$. 
By induction on the word length, we can show that 
$r_{21}$ is a multiple of $u$ in $\C[u, s_1^{\pm 1}, s_2^{\pm 1}]$. 
Hence we can write $r_{21}= u r'_{21}$ for some $r'_{21} \in \C[u, s_1^{\pm 1}, s_2^{\pm 1}]$. 
A word is said to be \textit{palindromic} if it reads the same backward or forward. 
By Lemma 1 in \cite{Ri} we have $$r_{22} - r_{11} + (s_1 - s_1^{-1})r_{12} = (s_2 - s_2^{-1}) r'_{21}$$
for any palindromic word $r$ of odd length.

The matrix equation $\rho(aw) = \rho(wa)$ is easily seen to be equivalent to the two equations 
$uw'_{21} =0$ and $w_{22} - w_{11} + (s_1 - s_1^{-1})w_{12} = (s_2 - s_2^{-1}) w'_{21}$.
Since $w$ is a palindromic word of odd length, we conclude that the matrix equation 
$\rho(aw) = \rho(wa)$ is equivalent to a single equation $w'_{21} = 0$. 
The polynomial $w'_{21}$ is called the Riley polynomial of a two-bridge link, and 
it determines the nonabelian representations of the link.

\begin{remark}
Other approaches, using character varieties and skein modules, to the Riley polynomial 
of a two-bridge link can be found in \cite{Qa, Tran-twobridgelink, LT}. 
\end{remark}

\subsection{Chebyshev polynomials} 

Recall that the $S_k(v)$ are the Chebychev polynomials of the second kind defined by 
$S_0(v)=1$, $S_1(v)=v$ and $S_{k}(v) = v S_{k-1}(v) - S_{k-2}(v)$ 
for all integers $k$. The following results are elementary, see e.g. \cite{Tr-volume}.

\begin{lemma} \label{chev1}
For any integer $k$ we have
$$S^2_k(v) + S^2_{k-1}(v) - v S_k(v) S_{k-1}(v) =1.$$
\end{lemma}

\begin{lemma} \label{chev2}
Suppose $V \in SL_2(\BC)$ and $v=\tr V$. For any integer $k$ we have
$$V^k = S_{k-1}(v) V - S_{k-2}(v) \mathbf{1}$$
where $\mathbf{1}$ denotes the $2 \times 2$ identity matrix.
\end{lemma}

We will need the following lemma in the proof of Theorem \ref{thm:A-polynomial}.

\begin{lemma} \label{expansion}
Suppose $v=2+q$. For $k \ge 0$ we have 
\begin{equation} \label{ind}
S_k(v) = \sum_{i=0}^k \binom{k+1+i}{2i+1} q^i.
\end{equation}
\end{lemma}

\begin{proof}
We use induction on $k \ge 0$. The cases $k=0,1$ are clear. 
Suppose $k \ge 2$ and \eqref{ind} holds true for $k-2$ and $k-1$. 
Since $S_k(v) = vS_{k-1}(v) - S_{k-2}(v)$, we have 
\begin{eqnarray*}
S_k(v) &=& (2+q) \sum_{i=0}^{k-1} \binom{k+i}{2i+1} q^i - \sum_{i=0}^{k-2} \binom{k-1+i}{2i+1} q^i \\
       &=& \sum_{i=0}^k \left\{ 2 \binom{k+i}{2i+1} + \binom{k+i-1}{2i-1} - \binom{k-1+i}{2i+1} \right\} q^i.
\end{eqnarray*}

It remains to show the following identity
$$\binom{k+1+i}{2i+1} = 2 \binom{k+i}{2i+1} + \binom{k+i-1}{2i-1} - \binom{k-1+i}{2i+1}.$$
This follows by applying the equality $\binom{c}{d} + \binom{c}{d+1} = \binom{c+1}{d+1}$ three times.
\end{proof} 

\subsection{Twisted Whitehead links}

In this subsection we compute nonabelian representations 
of twisted Whitehead links explicitly. 

We first consider the case of $W_{2n-1}$. Recall that 
the link group of $W_{2n-1}$ has a presentation $\pi_1(E_{W_{2n-1}}) = \la a, b \mid aw = wa \ra$, 
where $a,b$ are meridians depicted in Figure \ref{Wh} and 
$w=(bab^{-1}a^{-1})^na(a^{-1}b^{-1}ab)^n$. 

Suppose $\rho: \pi_1(E_{W_{2n-1}}) \to SL_2(\C)$ is a nonabelian representation. 
Up to conjugation, we may assume that $$\rho(a) = \left[ \begin{array}{cc}
s_1 & 1 \\
0 & s_1^{-1} \end{array} \right] \quad \text{and} \quad 
\rho(b) = \left[ \begin{array}{cc}
s_2 & 0 \\
u & s_2^{-1} \end{array} \right]$$
where $(u, s_1, s_2) \in \C^3$ satisfies the matrix equation $\rho(aw) = \rho(wa)$. 

Recall from the Introduction that $x = \tr \rho(a) = s_1 + s_1^{-1}$, $y = \tr \rho(b) = s_2 + s_2^{-1}$, 
$z = \tr \rho(ab) = u + s_1 s_2 + s_1^{-1} s_2^{-1}$ and 
$$v = \tr \rho(bab^{-1}a^{-1}) = u (u + s_1 s_2 + s_1^{-1} s_2^{-1} - s_1 s_2^{-1} - s_1^{-1} s_2) + 2.$$ 

Let $c= bab^{-1}a^{-1}$ and $d = a^{-1}b^{-1}ab$. We have 
$$\rho(c) =  \left[ \begin{array}{cc}
c_{11} & c_{12} \\
c_{21} & c_{22} \end{array} \right]
\qquad \text{and} \qquad 
\rho(d) = \left[ \begin{array}{cc}
d_{11} & d_{12} \\
d_{21} & d_{22} \end{array} \right]$$ 
where
\begin{eqnarray*}
c_{11} &=& 1 - s_1^{-1} s_2 u, \\
c_{12} &=& -s_1 + s_1 s_2^2 + s_2 u, \\
c_{21} &=& u (-s_1^{-2} s_2^{-1} + s_2^{-1} - s_1^{-1} u), \\
c_{22} &=& 1 + (s_1^{-1} s_2^{-1} - s_1 s_2^{-1} + s_1 s_2)u + u^2, \\
d_{11} &=& 1 + (s_1^{-1} s_2^{-1} - s_1^{-1} s_2 + s_1 s_2)u + u^2, \\
d_{12} &=& s_1^{-1} s_2^{-2} - s_1^{-1} + s_2^{-1} u, \\
d_{21} &=& u (s_2 - s_1^2 s_2 - s_1 u), \\
d_{22} &=& 1 - s_1 s_2^{-1} u.
\end{eqnarray*}

Since $v = \tr \rho(c) = \tr \rho(d)$, by Lemma \ref{chev2} we have
\begin{eqnarray*}
\rho(c^n) &=& \left[ \begin{array}{cc}
c_{11} S_{n-1}(v) -  S_{n-2}(v) & c_{12} S_{n-1}(v) \\
c_{21} S_{n-1}(v) & c_{22} S_{n-1}(v) -  S_{n-2}(v) \end{array} \right],\\ 
\rho(d^n) &=& \left[ \begin{array}{cc}
 d_{11}S_{n-1}(v) -  S_{n-2}(v) & d_{12} S_{n-1}(v) \\
d_{21} S_{n-1}(v) & d_{22} S_{n-1}(v) - S_{n-2}(v) \end{array} \right].
\end{eqnarray*}
This implies that $\rho(w) = \left[ \begin{array}{cc}
w_{11} & * \\
w_{21} & * \end{array} \right] = \rho(c^n a d^n)$ with 
\begin{eqnarray*}
w_{11} &=& (c_{12} d_{21} s_1^{-1} + c_{11} d_{21} + c_{11} d_{11} s_1) S^2_{n-1}(v) +  s_1 S^2_{n-2}(v) \\
       && \quad - \, (d_{21} + c_{11} s_1 + d_{11} s_1) S_{n-1}(v) S_{n-2}(v), \\
w_{21} &=& \big[ (c_{22} d_{21} s_1^{-1} + c_{21} d_{21} + c_{21} d_{11} s_1) S_{n-1}(v) 
- (d_{21} s_1^{-1} + c_{21} s_1) S_{n-2}(v)\big] S_{n-1}(v).
\end{eqnarray*}

By direct calculations we have
\begin{eqnarray*}
c_{22} d_{21} s_1^{-1} + c_{21} d_{21} + c_{21} d_{11} s_1   
      &=& u (xy - v z),\\
d_{21} s_1^{-1} + c_{21} s_1 
                             &=& u (xy - 2z).
\end{eqnarray*}
Hence $w_{21} = u w'_{21}$ where 
$$w'_{21}(W_{2n-1}) = \big( (xy - v z) S_{n-1}(v) - (xy - 2z)S_{n-2}(v) \big) S_{n-1}(v),$$
which is the Riley polynomial of $W_{2n-1}$.

Similarly, the Riley polynomial of $W_{2n}$ is given by the following formula
$$w'_{21}(W_{2n}) = \big( z S_{n}(v) - (xy -  z) S_{n-1}(v) \big) \big( S_{n}(v) - S_{n-1}(v) \big).$$ 

\begin{remark}
The above formulas for the Riley polynomials of twisted Whitehead links 
were already obtained in \cite{Tran-character} using character varieties. 
It was also shown in \cite{Tran-character} that the Riley polynomial
$w'_{21}(W_{2n-1}) \in \BC[x,y,z]$ is factored into exactly 
$n$ irreducible factors $(xy - v z) S_{n-1}(v) - (xy - 2z)S_{n-2}(v)$ and $v-2\cos \frac{j\pi}{n}$ $(1 \le j \le n-1)$. 
Similarly, $w'_{21}(W_{2n}) \in \BC[x,y,z]$ is factored into exactly 
$n+1$ irreducible factors $z S_{n}(v) - (xy -  z) S_{n-1}(v)$ and $v-2\cos \frac{(2j-1)\pi}{2n+1}$ $(1 \le j \le n)$. 
Note that $$S_{n-1}(v) = \prod_{j=1}^{n-1} \big( v-2\cos \frac{j\pi}{n} \big) \quad \text{and} \quad 
S_{n}(v) - S_{n-1}(v) = \prod_{j=1}^{n} \big( v-2\cos \frac{(2j-1)\pi}{2n+1} \big).$$
\end{remark}

\subsection{Proof of Theorem \ref{thm:A-polynomial}} 

We first consider the case of $W_{2n-1}$. 
The canonical longitudes corresponding to the meridians $a$ and $b$ are respectively
$\lambda_a = wa^{-1}$ and $\lambda_b = \overline{w} b^{-1}$, 
where $\overline{w}$ is the word obtained from $w$ by exchanging $a$ and $b$. 
Precisely, we have $\overline{w} = (aba^{-1}b^{-1})^nb(b^{-1}a^{-1}ba)^n$. 

Suppose $\rho: \pi_1(E_{W_{2n-1}}) \to SL_2(\C)$ is a nonabelian representation.  
With the same notations as in the previous subsection, we have 
$$\rho(w) = \left[ \begin{array}{cc}
w_{11} & * \\
0 & (w_{11})^{-1} \end{array} \right] \quad \text{and} \quad
\rho(\overline{w}) = \left[ \begin{array}{cc}
\overline{w}_{11} & 0 \\
* & (\overline{w}_{11})^{-1} \end{array} \right].$$

\begin{proposition} \label{components}
On $v-2\cos \frac{j\pi}{n} = 0$ we have
$$w_{11} = s_1 \quad \text{and} \quad \overline{w}_{11} = s_2.$$
\end{proposition}

\begin{proof}
Recall from the previous subsection that 
\begin{eqnarray*}
w_{11} &=& (c_{12} d_{21} s_1^{-1} + c_{11} d_{21} + c_{11} d_{11} s_1) S^2_{n-1}(v) +  s_1 S^2_{n-2}(v) \\
       && \quad - \, (d_{21} + c_{11} s_1 + d_{11} s_1) S_{n-1}(v) S_{n-2}(v).
\end{eqnarray*}
If $v=2\cos \frac{j\pi}{n}$ then $S_{n-1}(v) = 0$ then $S_{n-2}(v) = \pm 1$. Hence $w_{11} = s_1$.

The proof for $\overline{w}_{11} = s_2$ is similar.
\end{proof}

\begin{proposition} \label{canonical}
On $(xy - v z) S_{n-1}(v) - (xy - 2z)S_{n-2}(v)=0$ we have 
$$ 
w_{11} = \frac{-1-s_2^2+s_1 s_2 z}{s_1 + s_1 s_2^2-s_2 z} \quad \text{and} \quad 
\overline{w}_{11} = \frac{-1-s_1^2+s_1 s_2 z}{s_2 + s_1^2 s_2-s_1 z}.
$$
\end{proposition}

\begin{proof}
 By Lemma \ref{chev1} we have $S^2_{k}(v) + S^2 _{k-1}(v) - v S_{k}(v) S_{k-1}(v) = 1$ for all integers $k$.
Hence, on $(xy - v z) S_{n-1}(v) - (xy - 2z)S_{n-2}(v) = 0$ we have
$$S^2_{n-1}(v) = \left(  1 - r v + r^2\right)^{-1}$$ 
where $r = (xy - v z) / (xy - 2z)$. By a direct calculation we have
\begin{eqnarray*}
w_{11} &=& \big[ (c_{12} d_{21} s_1^{-1} + c_{11} d_{21} + c_{11} d_{11} s_1) + s_1 r^2 - (d_{21} + c_{11} s_1 + d_{11} s_1) r \big] \left(  1 - r v + r^2\right)^{-1} \\
       &=& \frac{s_1 s_2 (s_1^2 s_2 + s_1 u-s_2)}{s_1^2-s_1 s_2 u-1} = \frac{-1-s_2^2+s_1 s_2 z}{s_1 + s_1 s_2^2-s_2 z}.
\end{eqnarray*}
The proof for $\overline{w}_{11}$ is similar.
\end{proof}

We now prove Theorem \ref{thm:A-polynomial} for $W_{2n-1}$. 
Let $\rho(\lambda_a) = \left[ \begin{array}{cc}
L_a & * \\
0 & L_a^{-1} \end{array} \right]$. Since $\lambda_a = wa^{-1}$, 
we have $L_a = w_{11} s_1^{-1}$. On $$C_j : = \Big( \{v-2\cos \frac{j\pi}{n} = 0\} \cap \{s_2^2 = (\overline{w}_{11})^2 =1\} \Big)$$ 
by Proposition \ref{components} we have $w_{11} = s_1$. 
Hence $L_a - 1 = 0$ on $C_j$ for all $1\le j \le n-1$.

By Proposition \ref{canonical}, on  
$$C:= \Big( \{(xy - v z) S_{n-1}(v) - (xy - 2z)S_{n-2}(v)=0\} \cap \{s_2^2 = (\overline{w}_{11})^2 =1\} \Big)$$ we have 
$w_{11} = \frac{-1-s_2^2+s_1 s_2 z}{s_1 + s_1 s_2^2-s_2 z}$. Note that on $C$, we have $s_2^2=1$ implies $(\overline{w}_{11})^2 =1$.

With $s_2=1$ we then have $w_{11} = \frac{-2 + s_1 z}{2s_1-z}$. Hence $z = \frac{2(1 + s_1 w_{11})}{s_1 + w_{11}}$. 
Note that $v = 2 + (z-x)^2$. Let $z-x = t$. Then 
$$t = \frac{2(1 + s_1 w_{11})}{s_1 + w_{11}} - (s_1 + s_1^{-1}) 
= (s_1 - s_1^{-1})\big( \frac{w_{11} - s_1}{w_{11} + s_1} \big) 
= (s_1 - s_1^{-1})\Big( \frac{L_a - 1}{L_a + 1} \Big)$$
and $v = 2 + t^2$. Since $y = s_2 + s_2^{-1} =2$ and $v S_{n-1}(v) = S_n(v) + S_{n-2}(v)$ we have 
\begin{eqnarray*}
0 &=& 2x S_{n-1}(v) - z S_n(v) - (2x-z)S_{n-2}(v) \\
  &=& 2x S_{n-1}(v) - (x+t) S_n(v) - (x-t)S_{n-2}(v) \\
  &=& - x t^2 S_{n-1}(v) - t \big( S_n(v) - S_{n-2}(v) \big).
\end{eqnarray*}

If $t=0$ then $w_{11} = \frac{-2 + s_1 x}{2s_1-x} = s_1$. In this case we have $L_a = 1$. 
We now consider the case $x t S_{n-1}(v) + S_n(v) - S_{n-2}(v) =0$. 
Then, by Lemma \ref{expansion} we have 
\begin{eqnarray*}
0 &=& \sum_{i=0}^n \left\{ \binom{n+1+i}{2i+1} - \binom{n-1+i}{2i+1} \right\} t^{2i} 
+ x \sum_{i=0}^{n-1} \binom{n+i}{2i+1} t^{2i+1} \\
&=& \sum_{i=0}^n \left\{ \binom{n+1+i}{2i+1} - \binom{n-1+i}{2i+1} \right\} 
(s_1 - s_1^{-1})^{2i} \left( \frac{L_a -1}{L_a+1} \right)^{2i}\\
&& + \, \sum_{i=0}^{n-1} \binom{n+i}{2i+1} (s_1 + s_1^{-1}) (s_1 - s_1^{-1})^{2i+1} \left( \frac{L_a-1}{L_a+1} \right)^{2i+1}.
\end{eqnarray*} 
The last expression in the above equalities is exactly the polynomial $F(M_a, L_a)$, with $M_a=s_1$, defined in Theorem \ref{thm:A-polynomial}. 
Hence, with $s_2 = 1$, we have $(L_a - 1)F(M_a, L_a)=0$ on $C$. Similarly, with $s_2 = -1$, we obtain the same equation 
$(L_a - 1)F(M_a, L_a)=0$ on $C$. This proves the formula of the A-polynomial for 
the component of $W_{2n-1}$ corresponding to the meridian $a$. 
The one for the component corresponding $b$ is exactly the same.

For $W_{2n}$ and its component corresponding to $a$, we have $L_a M_a^2 - 1 = 0$ on  
$$D_j : = \Big( \{v-2\cos \frac{(2j-1)\pi}{2n+1} = 0\} \cap \{s_2^2 = (\overline{w}_{11})^2 =1\} \Big)$$ 
for $1 \le j \le n$. Moreover, on 
$$D:= \Big( \{z S_{n}(v) - (xy -  z) S_{n-1}(v)=0\} \cap \{s_2^2 = (\overline{w}_{11})^2 =1\} \Big)$$ 
we have $G(M_a, L_a) = 0$ where $G$ is the polynomial defined in Theorem \ref{thm:A-polynomial}.
The same formulas can be obtained for the component of $W_{2n}$ corresponding to $b$.

This completes the proof of Theorem \ref{thm:A-polynomial}.

\subsection{Proof of Theorem \ref{thm:canonical}} \label{Tr-canonical}

We make use of Proposition \ref{prop:hyp} which implies that for each $j=1,2$ 
the factor $A_j^{\text{can}}(M,L)$ of the A-polynomial $A_j(L,M)$ corresponding to a canonical component of $W_k$ (with $k \ge 1$)
has an irreducible factor containing at least 3 monomials in $M, L$. 
This irreducible factor cannot be $L-1$ or $LM^2-1$. 
Hence, from the proof of Theorem \ref{thm:A-polynomial} we conclude that 
the canonical components of $W_{2n-1}$ and $W_{2n}$ are respectively the zero sets of the polynomials 
$(xy - v z) S_{n-1}(v) - (xy - 2z)S_{n-2}(v)$ and
$z S_{n}(v) - (xy -  z) S_{n-1}(v)$. This completes the proof of Theorem \ref{thm:canonical}.

\subsection{Proof of Theorem \ref{thm:volume}}

Recall that $E_{W_k}(\alpha)$ is the cone-manifold of $W_k$ with cone angles $\alpha_1 = \alpha_2 = \alpha$. 
There exists an angle $\alpha_{W_k} \in [\frac{2\pi}{3},\pi)$ such that $E_{W_k}(\alpha)$ is hyperbolic for $\alpha \in (0, \alpha_{W_k})$, Euclidean for $\alpha =\alpha_{W_k}$, and spherical for $\alpha \in (\alpha_{W_k}, \pi)$.

We first consider the case of $W_{2n-1}$. 
By Theorem \ref{thm:canonical} the canonical component of $W_{2n-1}$ is determined by 
$(xy - v z) S_{n-1}(v) - (xy - 2z)S_{n-2}(v)=0$. 
Moreover, by Proposition \ref{canonical}, on this component we have 
$$ 
w_{11} = \frac{-1-s_2^2+s_1 s_2 z}{s_1 + s_1 s_2^2-s_2 z}.
$$
Here we use the same notations as in the previous subsections.

For $\alpha \in (0, \alpha_{W_{2n-1}})$, by the Schlafli formula we have 
$$\text{Vol} \, E_{W_{2n-1}}(\alpha) 
= \int_{\alpha}^{\pi} 2\log \left| w_{11} \right| d\omega 
= \int_{\alpha}^{\pi} 2\log \left| \frac{-1-s_2^2+s_1 s_2 z}{s_1 + s_1 s_2^2-s_2 z} \right| d\omega$$
where $s_1 = s_2 = s = e^{i\omega/2}$ and $z$ (with $\big| \frac{-1-s_2^2+s_1 s_2 z}{s_1 + s_1 s_2^2-s_2 z} \big| \ge 1$) satisfy $$(xy - v z) S_{n-1}(v) - (xy - 2z)S_{n-2}(v)=0.$$
We refer the reader to \cite{DMM} for the volume formula of hyperbolic cone-manifolds of links using the Schlafli formula.

Note that $\big| \frac{-1-s_2^2+s_1 s_2 z}{s_1 + s_1 s_2^2-s_2 z} \big| = \big| \frac{z - (s^{-2}+1)}{z- (s^2+1) } \big| \ge 1$ is equivalent to $\text{Im} \,z \ge 0$. 
Since $x=y=s+s^{-1}$ and $v= x^2 + y^2 + z^2 - xyz -2 = 2x^2+z^2-x^2z-2$, 
the proof of Theorem \ref{thm:volume} for $W_{2n-1}$ is complete. The proof for $W_{2n}$ is similar.
\section*{Acknowledgements} 
The author has been partially supported by a grant from the Simons Foundation (\#354595 to Anh Tran).


\end{document}